\documentclass[12pt,fleqn]{article}
\usepackage{amsfonts}
\usepackage{mathrsfs}
\usepackage[dvips]{graphics}
\usepackage[dvips]{color}
\usepackage{amsthm}
\usepackage[]{amsmath}
\usepackage{CJK}
\newtheorem{proposition}{Proposition}[section]
\newtheorem{lemma}[proposition]{Lemma}
\newtheorem{definition}[proposition]{Definition}
\newtheorem{Theorem}[proposition]{Theorem}
\newtheorem{corollary}[proposition]{Corollary}
\newtheorem{property}[proposition]{Property}
\begin{document}
\begin{CJK*}{GBK}{song}

\centerline{\Large{\textbf{Envelope Word and Gap Sequence in Doubling Sequence}}}

\vspace{0.1cm}


\centerline{Yuke Huang\footnote[1]{Department of Mathematical Sciences, Tsinghua University, Beijing, 100084, P. R. China.}$^,$\footnote[2]{E-mail address: hyg03ster@163.com.}
~~Hanxiong Zhang\footnote[3]{Department of Mathematics, China University of Mining and Technology, Beijing, 100083, P. R. China.}$^,$\footnote[4]{E-mail address: zhanghanxiong@163.com(Corresponding author).}}

\vspace{1cm}

\centerline{\textbf{\large{ABSTRACT}}}

\vspace{0.4cm}

\noindent Let $\omega$ be a factor of Doubling sequence $D_\infty=x_1x_2\cdots$, then it occurs in the sequence infinitely many times. Let $\omega_p$ be the $p$-th occurrence of $\omega$ and $G_p(\omega)$ be the gap between $\omega_p$ and $\omega_{p+1}$.
In this paper, we discuss the structure of the gap sequence $\{G_p(\omega)\}_{p\geq1}$.
We prove that all factors can be divided into two types, one type has exactly two distinct gaps $G_1(\omega)$ and $G_2(\omega)$, the other type has exactly three distinct gaps $G_1(\omega)$, $G_2(\omega)$ and $G_4(\omega)$. We determine the expressions of gaps completely. And also give the substitution of each gap sequence.
The main tool in this paper is ``envelope word", which is a new notion, denoted by $E_{m,i}$.
As an application, we determine the positions of all $\omega_p$, discuss some combinatorial properties of factors, and count the distinct squares beginning in $D_\infty[1,N]$ for $N\geq1$.



\vspace{1cm}

\setcounter{section}{1}

\noindent\textbf{\large{1.~Introduction}}

\vspace{0.4cm}

Factor property have been studied extensively, such as Lothaire\cite{L1983,L2002}. Doubling sequence, sometimes appears in a different name: period-doubling sequence, is a
classical example of a substitution of constant length over a binary alphabet.
It is also the difference of Thue-Morse sequence \cite{B1989,P1981}.
The Doubling sequence, among other so-called primitive substitution sequence, has many remarkable properties, and it appears in many aspects of mathematics and computer science, symbolic dynamics, theoretical computer science etc., we refer to Allouche and Shallit\cite{AS2003}, Berstel\cite{B1980} and also Mauduit\cite{M2001}.

Wen and Wen\cite{WW1994} studied the factor structure of Fibonacci sequence, where they defined the singular word and give the positively separate property of the singular words.
Huang and Wen\cite{HW2014} extend the results from singular words to arbitrary words $\omega$ of the Fibonacci sequence, and discuss the structure of gap sequence $\{G_p(\omega)\}_{p\geq1}$.
The main aim of this article is to extend the results in Huang and Wen\cite{HW2014} from Fibonacci sequence to Doubling sequences $D_\infty$.

The main result in this paper is as follows.

\vspace{0.2cm}

\noindent\textbf{Theorem} (Gap sequence of factor $\omega\prec D_\infty$)\textbf{.}

\emph{(1) All factors can be divided into two types;}

\emph{(2) One type of factor has exactly two distinct gaps $G_1(\omega)$ and $G_2(\omega)$, the gap sequence is $\varphi_1(D_\infty)$, where $\varphi_1(a,b)=(a,bb)$;}

\emph{(3) The other type of factor has exactly three distinct gaps $G_1(\omega)$, $G_2(\omega)$ and $G_4(\omega)$, the gap sequence is $\varphi_2(D_\infty)$, where $\varphi_2(a,b)=(ab,acac)$.}

\vspace{0.2cm}

The main tool in this paper is ``envelope word". Using it, we can determine the type of each factor, and give the expressions of each gap $G_p(\omega)$ and each substitution $\varphi_i$.
Then we can determine the structure of gap sequence of $D_\infty$ completely.

\vspace{0.2cm}

This paper is organized as follows.

Section 1 is devoted to the introduction and preliminaries. In Section 2, we give some basic properties of Doubling sequence. In Section 3,
we define a new notion ``envelope word" and give the weak version of ``uniqueness of envelope extension property", there are two types of envelope words.
In Section 4, we discuss the gaps and gap sequence of envelope words.
Then we give the strong version of ``uniqueness of envelope extension property" in Section 5, which makes envelope words quite special. Using them, we determine the gaps and gap sequence of an arbitrary word.
As an application, we determine the positions of $\omega_p$ for all $(\omega,p)$ in Section 6. In Section 7, we study some combinatorial properties of factors in Doubling sequence.

\vspace{0.4cm}

\noindent\emph{1.1 Notation and Basic Properties}

\vspace{0.4cm}

Let $\mathcal{A}=\{a,b\}$ be a binary alphabet. Let $\mathcal{A}^\ast$ be the set of finite words on $\mathcal{A}$ and $\mathcal{A}^{\mathbb{N}}$ be the set of one-sided infinite words. The elements of $\mathcal{A}^\ast$ are called words or factors, which will be denoted by $\omega$. The neutral element of $\mathcal{A}^\ast$ is called the empty word, which we denote by $\varepsilon$. For a finite word $\omega=x_1x_2\cdots x_n$, the length of $\omega$ is equal to $n$ and denoted by $|\omega|$.
The number of occurrences of letter $\alpha\in\mathcal{A}$ in $\omega$ is denoted by $|\omega|_\alpha$.

The concatenation of two words $\nu=x_1x_2\cdots x_r$ and $\omega=y_1y_2\cdots y_n$ is the word $\nu\omega=x_1x_2\cdots x_ry_1y_2\cdots y_n$. This operation is associative and has a unit element, the empty word $\varepsilon$.

As usual, the Doubling sequence $D_\infty$ is defined as the fixed point beginning with letter $a$ under the Doubling substitution $\sigma$, which is defined over the alphabet $\mathcal{A}=\{a,b\}$ by $\sigma(a)=ab$ and $\sigma(b)=aa$:
$$D_\infty=\sigma^\infty(a)=abaaabab abaaabaa abaaabababaaabab\cdots$$

We also use the notations that $A_m=\sigma^m(a)$ and $B_m=\sigma^m(b)$ where $m\geq0$.
Obviously, $|A_m|=|B_m|=2^m$.

The notation $\nu\prec\omega$ means that word $\nu$ is a factor of word $\omega$.

Let $\tau=x_1x_2\cdots x_m\cdots$ be a sequence. For any $i\leq j$, define $\tau[i,j]:=x_ix_{i+1}\cdots x_{j-1}x_j$, the factor of $\tau$ of length $j-i+1$, starting from the $i$-th letter and ending to the $j$-th letter. By convention, we note $\tau[i]:=\tau[i,i]=x_i$ and $\tau[i,i-1]:=\varepsilon$.

If $\omega=\tau[i,j]$, $\tau$ is a word or a sequence,
$\omega$ is said to be occur at position $i$ in $\tau$. Furthermore, $L(\omega,p)$ denotes the position of $\omega_p$.

Let $\nu=\nu_1\nu_2\cdots\nu_n\in\mathcal{A}^\ast$, we denote by $\nu^{-1}:=\nu_n^{-1}\cdots\nu_2^{-1}\nu_1^{-1}$, called the inverse word of $\nu$.
If $\omega=u\nu$, then $\omega^{-1}=(u\nu)^{-1}=\nu^{-1}u^{-1}$.
Furthermore, $\omega\nu^{-1}=u\nu\nu^{-1}=u$ and $u^{-1}\omega=u^{-1}u\nu=\nu$.

\vspace{0.4cm}

\noindent\emph{1.2 Factor Sequence and Gap Sequence}

\vspace{0.4cm}

Let $\omega$ be a factor of Doubling sequence $D_\infty$. In this subsection, we will introduce some definitions: factor sequence $\{\omega_p\}_{p\ge 1}$, gap word $G_p(\omega)$, gap sequence $\{G_p(\omega)\}_{p\ge 1}$, etc.

\begin{definition}[Factor sequence]
Let $\omega$ be a factor of Doubling sequence, then it occurs in the sequence infinitely many times, which we arrange by the sequence
$\{\omega_p\}_{p\ge 1}$, where $\omega_p$ denote the $p$-th occurrence of $\omega$.
\end{definition}

\begin{definition}[Gap]
Let $\omega_p=x_{i+1}\cdots x_{i+n}$, $\omega_{p+1}=x_{j+1}\cdots x_{j+n}$, the gap
between $\omega_p$ and $\omega_{p+1}$, denoted by $G_p(\omega)$, is defined by
\begin{equation*}
G_p(\omega)=
\begin{cases}
\varepsilon&when~i+n=j,~\omega_p~and~\omega_{p+1}~are~adjacent;\\
x_{i+n+1}\cdots x_{j}&when~i+n<j,~\omega_p~and~\omega_{p+1}~are~separated;\\
(x_{j+1}\cdots x_{i+n})^{-1}&when~i+n>j,~\omega_p~and~\omega_{p+1}~are~overlapped.
\end{cases}
\end{equation*}

The set of gaps of factor $\omega$ is defined as $\{G_p(\omega)|~p\geq1\}$.
\end{definition}


\noindent\textbf{Example.} $G_1(aa)=a^{-1}$ (overlapped), $G_2(aa)=babab$ (separated), $G_2(aab)=\varepsilon$ (adjacent).

\vspace{0.2cm}

\noindent\textbf{Remark.}
A closely related concept of ``Gap" is ``Return Word", introduced by F.Durand\cite{D1998}, for characterizing a sequence over a finite alphabet to be substitutive, he proved that a sequence is primitive substitutive if and only if the set of its return words is finite, which means, each factor of this sequence has finite return words.
In 2001, L.Vuillon\cite{V2001} proved that an infinite word $\tau$ is a Sturmian sequence if and only if each non-empty factor $\omega\prec\tau$ has exactly two distinct return words.
Some other related researches (see also \cite{BPS2008}) focused in the cardinality of the set of return words of $\omega$ and the consequent results, but didn't concern about the structures of the sequences derived by return words.

Essentially gap words can be derived from the return words by deleting one prefix $\omega$, but since the terminology ``gap" will be convenient and have some advantages for our discussions, we prefer to adopt it.

\begin{definition}[Gap sequence]
Let $G_p(\omega)$ be the gap between $\omega_p$ and $\omega_{p+1}$, we call $\{G_p(\omega)\}_{p\ge 1}$ the gap sequence of the factor $\omega$.
\end{definition}


\vspace{0.5cm}

\stepcounter{section}

\noindent\textbf{\large{2.~Some Basic Properties of Doubling Sequence}}

\vspace{0.4cm}

In this section, we will give some very important properties of Doubling sequence $D_\infty$.

\begin{definition}[$\delta_m$] Let $\delta_m$ be the last letter of $A_m$.
\end{definition}

\begin{property} (1) $\delta_{m-1}$ be the last letter of $B_m$;

(2) For $m\geq0$, $\delta_m=a$ (resp. $b$) when $m$ is even (resp. odd).
\end{property}

\begin{proposition} $A_m\delta_m^{-1}=B_m\delta_{m+1}^{-1}$ for $m\geq0$.  \end{proposition}

\begin{property}[Palindrome] $A_m\delta_m^{-1}$ is a palindrome for $m\geq0$.\end{property}

\begin{proof} By induction, when $m=0,1,2$, $A_0\delta_0^{-1}=\varepsilon$, $A_1\delta_1^{-1}=a$, $A_2\delta_2^{-1}=aba$. Assume $A_{m}\delta_m^{-1}$ is a palindrome.
Since $A_{m+2}\delta_m^{-1}=A_{m}B_{m}A_{m}A_{m}\delta_m^{-1}$ and
$B_m\delta_{m-1}^{-1}=A_m\delta_m^{-1}$,
$$A_{m+2}\delta_m^{-1}
=A_{m}\delta_m^{-1}\ast \delta_m\ast A_{m}\delta_m^{-1}\ast \delta_{m-1}\ast A_{m}\delta_m^{-1}\ast \delta_m\ast A_{m}\delta_m^{-1}.$$
That means $A_{m+2}\delta_m^{-1}$ is a palindrome too.
\end{proof}

\begin{lemma}[Positions] For $m\geq1$,

(1) $A_m$ occurs in $A_mB_mA_m$ twice at positions 1 and $2^{m+1}+1$;

(2) $A_m$ occurs in $A_mA_m$ twice at positions 1 and $2^m+1$;

(3) $B_m$ occurs in $B_mB_{m-1}B_{m+1}B_m$ twice at positions 1 and $7\ast2^{m-1}+1$;

(4) $B_m$ occurs in $B_mB_{m-1}B_m$ twice at positions 1 and $3\ast2^{m-1}+1$.
\end{lemma}

These properties above will be used in the proofs of Theorem 4.1, 4.2 and 4.3.


\vspace{0.5cm}

\stepcounter{section}

\noindent\textbf{\large{3.~Envelope Word $E_{m,i}$}}

\vspace{0.4cm}

In this section, we will discuss a particular kind of factor, called envelope word, denoted by $E_{m,i}$. We will introduce some definitions: envelope word $E_{m,i}$, envelope set $\mathcal{E}$, envelope word of factor $\omega$, denoted by $Env(\omega)$, etc.

\begin{definition}[Envelope set and envelope word in Doubling sequence]\

Envelope set $\mathcal{E}:=\{E_{m,i},~m\geq1,~i=1,2\}$, where
$$E_{m,1}=A_m\delta_m^{-1},~E_{m,2}=A_{m-1}A_m\delta_m^{-1}.$$
The element $E_{m,i}$ in $\mathcal{E}$ is called an envelope word in Doubling sequence.
\end{definition}

\begin{property} For $m\geq1$ and $i=1,2$, $E_{m+1,i}=\sigma(E_{m,i})a$.
\end{property}

\noindent\textbf{Example.} $E_{1,1}=a$, $E_{2,1}=aba$, $E_{3,1}=abaaaba
$; $E_{1,2}=aa$, $E_{2,2}=ababa$, $E_{3,2}=abaaabaaaba$.

\begin{property} For $m\geq1$,
$E_{m+1,1}=E_{m,1}\delta_{m}E_{m,1}$ and
$E_{m+1,2}=E_{m,1}\delta_{m}E_{m,1}\delta_{m}E_{m,1}$.
\end{property}

\begin{property}[Palindrome] Envelope word $E_{m,i}$ is a palindrome for $m\geq1$, $i=1,2$.\end{property}
%

\begin{definition}[Order in Envelope Set $\mathcal{E}$]\

Let $E_{m_1,i_1}$ and $E_{m_2,i_2}$ be two elements in envelope set $\mathcal{E}$.

(1) if $m_1<m_2$ then $E_{m_1,i_1}\sqsubset E_{m_2,i_2}$;

(2) if $m_1=m_2$ and $i_1<i_2$ then $E_{m_1,i_1}\sqsubset E_{m_2,i_2}$.
\end{definition}

\noindent\textbf{Example.} $E_{1,1}\sqsubset E_{1,2}\sqsubset E_{2,1}$.

\begin{definition}[Envelope word of factor $\omega$]\

The envelope word of factor $\omega\in D_\infty$, denoted by $Env(\omega):=E_{m,i}$, is the minimal element in envelope set $\mathcal{E}$ under the order $"\sqsubset"$ subject to $\omega\prec E_{m,i}$.
\end{definition}

\noindent\textbf{Example.} $Env(bab)=ababa=E_{2,2}$, $Env(aaa)=abaaaba=E_{3,1}$.

\begin{Theorem}[Uniqueness of envelope extension, weak version]\

Let $Env(\omega)=E_{m,i}$, $i=1,2$, then $\omega$ has an unique envelope extension as
$$E_{m,i}=\mu_1(\omega)\ast \omega\ast\mu_2(\omega),$$
where $\mu_1(\omega)$ and $\mu_2(\omega)$ are constant words depending only on $\omega$.
\end{Theorem}

\begin{proof}
\textbf{Part 1.} When $i=1$, $E_{m,1}=A_m\delta_m^{-1}$. 
It contains $E_{m-1,1}$ and $E_{m-2,2}$ as subwords.
\setlength{\unitlength}{1mm}
\begin{center}
\begin{picture}(125,40)
\put(0,10){$E_{m,1}$}
\put(10,10){$=$}
\put(15,10){$A_{m-2}$}
\put(25,10){$A_{m-3}$}
\put(35,10){$\delta_{m-1}^{-1}$}
\put(43,10){$\delta_{m-1}$}
\put(51,10){$A_{m-3}$}
\put(61,10){$\delta_{m-1}^{-1}$}
\put(69,10){$\delta_{m-1}$}
\put(77,10){$A_{m-2}$}
\put(87,10){$\delta_{m-2}^{-1}$}
\put(95,10){$\delta_{m-2}$}
\put(103,10){$A_{m-2}$}
\put(113,10){$\delta_{m-2}^{-1}$}
\put(15,8){\line(1,0){54}}
\put(15,8){\line(0,1){2}}
\put(69,8){\line(0,1){2}}
\put(42,6){\line(0,1){2}}
\put(24,2){$A_{m-1}\delta_{m-1}^{-1}=E_{m-1,1}$}
\put(77,8){\line(1,0){44}}
\put(77,8){\line(0,1){2}}
\put(121,8){\line(0,1){2}}
\put(101,6){\line(0,1){2}}
\put(83,2){$A_{m-1}\delta_{m-1}^{-1}=E_{m-1,1}$}
\put(51,15){\line(1,0){44}}
\put(51,15){\line(0,-1){2}}
\put(95,15){\line(0,-1){2}}
\put(74,15){\line(0,1){2}}
\put(51,18){$A_{m-3}A_{m-2}\delta_{m-2}^{-1}=E_{m-2,2}$}
\put(43,24){\line(1,0){34}}
\multiput(43,24)(0,-3){4}{\line(0,-1){2}}
\put(77,24){\line(0,-1){2}}
\put(60,24){\line(0,1){2}}
\put(46,27){$\upsilon_1=\delta_{m-1}A_{m-3}$}
\put(69,33){\line(1,0){34}}
\put(69,33){\line(0,-1){2}}
\multiput(103,33)(0,-3){7}{\line(0,-1){2}}
\put(86,33){\line(0,1){2}}
\put(74,36){$\upsilon_2=\delta_{m-1}A_{m-2}$}
\end{picture}
\end{center}
\centerline{Fig. 3.1: The $E_{m-1,1}$ and $E_{m-2,2}$ in $E_{m,1}$.}

\vspace{0.2cm}

By the definition of $Env(\omega)=E_{m,1}$, $\omega$ must contain $\upsilon_1=\delta_{m-1}A_{m-3}$ or $\upsilon_2=\delta_{m-1}A_{m-2}$ as a subword, i.e. $\delta_{m-1}A_{m-3}\prec\omega$.
$$E_{m,1}=A_{m-3}B_{m-3}A_{m-3}A_{m-3}A_{m-3}B_{m-3}A_{m-3}A_{m-3}\delta_{m-3}^{-1},$$
so all possible positions of $\delta_{m-3}A_{m-3}$ in $E_{m,1}$ is $3\ast2^{m-3}$ and $2^{m-1}$.

Suppose factor $\omega$ occurs twice in $E_{m,1}$, then factor $\delta_{m-3}A_{m-3}$ in $\omega$ occurs at the two positions above.
Since the two $\omega$ are equal, consider the structure around the two $\delta_{m-3}A_{m-3}$:
\setlength{\unitlength}{1mm}
\begin{center}
\begin{picture}(150,22)
\put(0,10){$E_{m,1}$}
\put(10,10){$=$}
\put(15,10){$A_{m-2}$}
\put(25,10){$A_{m-3}$}
\put(35,10){$\delta_{m-1}^{-1}$}
\put(43,10){$\delta_{m-1}$}
\put(51,10){$A_{m-3}$}
\put(61,10){$\delta_{m-1}^{-1}$}
\put(69,10){$\delta_{m-1}$}
\put(77,10){$A_{m-3}$}
\put(87,10){$\delta_{m-3}^{-1}$}
\put(95,10){$\delta_{m-3}$}
\put(103,10){$A_{m-3}$}
\put(113,10){$\delta_{m-3}^{-1}$}
\put(121,10){$\delta_{m-2}$}
\put(129,10){$A_{m-2}$}
\put(139,10){$\delta_{m-2}^{-1}$}
\linethickness{1.5pt}
\put(43,8){\line(1,0){34}}
\put(69,15){\line(1,0){18}}
\linethickness{0.5pt}
\put(43,8){\line(0,1){2}}
\put(77,8){\line(0,1){2}}
\put(60,8){\line(0,-1){2}}
\put(46,2){$\upsilon_1=\delta_{m-1}A_{m-3}$}
\put(77,8){\line(1,0){18}}
\put(43,8){\line(-1,0){18}}
\put(69,15){\line(0,-1){2}}
\put(87,15){\line(0,-1){2}}
\put(78,15){\line(0,1){2}}
\put(64,18){$\upsilon_1=\delta_{m-1}A_{m-3}$}
\put(87,15){\line(1,0){34}}
\put(69,15){\line(-1,0){18}}
\end{picture}
\end{center}
\centerline{Fig. 3.2: The structure around the two $\delta_{m-3}A_{m-3}$.}
\vspace{0.2cm}

So we get that
$$\omega\prec A_{m-3}A_{m-3}A_{m-3}\delta_{m-3}^{-1}=A_{m-3}A_{m-2}\delta_{m-2}^{-1}=E_{m-2,2}\sqsubset E_{m,1}.$$
It contradicts the definition of $Env(\omega)=E_{m,1}$.

This means $\omega$ occurs in $E_{m,1}$ only once, i.e. the envelope extension is unique.

\vspace{0.2cm}

\textbf{Part 2.} When $i=2$, $E_{m,2}=A_{m-1}A_m\delta_m^{-1}$.
It contains $E_{m,1}$ as subwords.
\setlength{\unitlength}{1mm}
\begin{center}
\begin{picture}(90,32)
\put(0,10){$E_{m,2}$}
\put(10,10){$=$}
\put(15,10){$A_{m-1}$}
\put(25,10){$\delta_{m-1}^{-1}$}
\put(33,10){$\delta_{m-1}$}
\put(41,10){$A_{m-1}$}
\put(51,10){$\delta_{m-1}^{-1}$}
\put(59,10){$\delta_{m-1}$}
\put(67,10){$B_{m-1}$}
\put(77,10){$\delta_{m}^{-1}$}
\put(15,8){\line(1,0){44}}
\put(15,8){\line(0,1){2}}
\put(59,8){\line(0,1){2}}
\put(37,6){\line(0,1){2}}
\put(24,2){$A_{m}\delta_{m}^{-1}=E_{m,1}$}
\put(41,15){\line(1,0){42}}
\put(41,15){\line(0,-1){2}}
\put(83,15){\line(0,-1){2}}
\put(62,15){\line(0,1){2}}
\put(50,18){$A_{m}\delta_{m}^{-1}=E_{m,1}$}
\put(33,24){\line(1,0){34}}
\multiput(33,24)(0,-3){4}{\line(0,-1){2}}
\put(67,24){\line(0,-1){2}}
\put(50,24){\line(0,1){2}}
\put(42,27){$\delta_{m-1}A_{m-1}$}
\end{picture}
\end{center}
\centerline{Fig. 3.3: The $E_{m,1}$ occurs in $E_{m,2}$.}

\vspace{0.2cm}

By the definition of $Env(\omega)=E_{m,2}$, $\omega$ must contain $\delta_{m-1}A_{m-3}$ or $\upsilon_2=\delta_{m-1}A_{m-1}$ as a subword, i.e. $\delta_{m-1}A_{m-1}\prec\omega$.
$$E_{m,2}=A_{m-2}B_{m-2}A_{m-2}B_{m-2}A_{m-2}A_{m-2}\delta_m^{-1}.$$
So all possible positions of $\delta_{m-1}\ast A_{m-1}$ in $E_{m,1}$ is $2^{m-1}$.

This means $\omega$ occurs in $E_{m,2}$ only once, i.e. the envelope extension is unique.
\end{proof}

\begin{lemma}[Positions] For $m\geq1$,

(1) $E_{m,1}$ occurs in $E_{m+1,1}$ twice at positions 1 and $2^m+1$;

(2) $E_{m,1}$ occurs in $E_{m,2}$ twice at positions 1 and $2^{m-1}+1$;

(3) $E_{m,2}$ occurs in $E_{m+2,1}$ twice at positions $2^m+1$ and $3\ast2^{m-1}+1$;

(4) $E_{m,2}$ occurs in $E_{m+2,2}$ four times:
$2^m+1$, $3\ast2^{m-1}+1$, $3\ast2^m+1$ and $7\ast2^{m-1}+1$.
\end{lemma}

\begin{lemma} For $m\geq1$, $i=1,2$,

(1) $E_{m,1}\prec E_{m',i}$, $\forall m'\geq m$;

(2) $E_{m,2}\not\prec E_{m+1,i}$;

(3) $E_{m,2}\prec E_{m'+2,i}$, $\forall m'\geq m$.
\end{lemma}

\begin{property}[Envelope set $\mathcal{E}$ is minimum]\

Envelope set $\mathcal{E}$ is minimum to ensure Theorem 3.7.
\end{property}

\begin{proof} Let $\mathcal{E}'$ be a proper subset of $\mathcal{E}$. $E_{m,i}$ is the minimal element of $\mathcal{E}-\mathcal{E}'$. Let
$$\mathcal{C}=\{E_{n,j}\in\mathcal{E}':E_{m,i}\prec E_{n,j}\}.$$
If $\mathcal{C}=\emptyset$, $E_{m,i}$ does not have envelope word.
It contradicts the definition of envelope word. 

If $\mathcal{C}\neq\emptyset$, $Env(E_{m,i})=\min\{E_{n,j}\in\mathcal{E}':E_{m,i}\prec E_{n,j}\}$, denoted by $E_{m',i'}$.

Case 1: $i=1$ and $m'>m$. By Lemma 3.9(1), $E_{m+1,1}\prec E_{m',i'}$. By Lemma 3.8(1), $E_{m,1}$ occurs in $E_{m+1,1}$ twice. So $E_{m,1}$ occurs in $Env(E_{m,1})$ more than once.

Case 2: $i=1$ and $m'=m$, then $E_{m',i'}=E_{m,2}$. By Lemma 3.8(2), $E_{m,1}$ occurs in $Env(E_{m,1})$ more than once.

Case 3: $i=2$. By Lemma 3.9(2), $m'>m+1$. By Lemma 3.9(1), $E_{m+2,1}\prec E_{m',i'}$. By Lemma 3.8(3), $E_{m,2}$ occurs in $E_{m+2,1}$ twice. So $E_{m,2}$ occurs in $Env(E_{m,1})$ more than once.

All cases contradict the weak version of ``uniqueness of envelope extension" in Theorem 3.7. So envelope set $\mathcal{E}$ is minimum to ensure Theorem 3.7.
\end{proof}


\vspace{0.5cm}

\stepcounter{section}

\noindent\textbf{\large{4.~Gaps and Gap Sequence of Envelope Word $E_{m,i}$}}

\vspace{0.4cm}

Each envelope word $E_{m,i}$ occurs in Doubling sequence infinitely many times. Let $E_{m,i,p}$ be the $p$-th occurrence of $E_{m,i}$, $G_p(E_{m,i})$ be the gap between $E_{m,i,p}$ and $E_{m,i,p+1}$, $G_0(E_{m,i})$ be the prefix of $D_\infty$ before $E_{m,i,1}$.

\begin{Theorem} Let $G_0(E_{m,i})$ be the prefix of $D_\infty$ before $E_{m,i,1}$, then

(1) $G_0(E_{m,1})=D_\infty[1,0]=\varepsilon$;

(2) $G_0(E_{m,2})=D_\infty[1,2^m]=A_m$.
\end{Theorem}

\begin{proof} (1) Since $D_\infty[1,2^m-1]=A_m\delta_m^{-1}=E_{m,1}$, $G_0(E_{m,1})=\varepsilon$.

(2) Since $D_\infty[1,5\ast2^{m-1}-1]=A_{m-1}B_{m-1}A_{m-1}A_{m-1}A_{m-1}\delta_{m+1}
=A_{m}A_{m-1}B_{m}\delta_{m-1}$. By Proposition 2.3, $B_{m}\delta_{m+1}=A_{m}\delta_m$,
so $D_\infty[2^m+1,5\ast2^{m-1}-1]=A_{m-1}A_m\delta_m^{-1}=E_{m,2}$.

Consider the first $A_{m-1}$ in $\underline{A_{m-1}}A_m\delta_m^{-1}$. By Lemma 2.5(1), it
occurs in $A_{m-1}B_{m-1}A_{m-1}$ twice at positions 1 and $2^{m}+1$.
Suppose this $A_{m-1}$ occurs at positions 1, then the $2^m$-th letter in $A_{m-1}B_{m-1}A_{m-1}$ is $\delta_m$. But the letter in $A_{m-1}A_m\delta_m^{-1}=A_{m-1}A_{m-1}B_{m-1}\delta_m^{-1}$ at the same position is $\delta_{m-1}$. It contradicts $\delta_m\neq\delta_{m-1}$.

This means $D_\infty[2^m+1,5\ast2^{m-1}-1]=E_{m,2,1}$ and $G_0(E_{m,2})=A_m$.
\end{proof}

\begin{Theorem}[The first two gaps of $E_{m,1}$] For $m\geq1$,

(1) $G_1(E_{m,1})=\delta_m$;

(2) $G_2(E_{m,1})=\delta_{m-1}A_{m-1}^{-1}$.
\end{Theorem}

\begin{proof} By the definition of envelope word, $E_{m,1}=A_m\delta_m^{-1}$. 
By Property 2.3, $B_m\delta_{m+1}^{-1}=A_m\delta_{m}$,
$$D_\infty[1,2^{m+1}-1]=A_mB_m\delta_{m-1}^{-1}=A_mA_m\delta_{m}^{-1}
=A_m\delta_{m}^{-1}\ast\delta_{m}\ast A_m\delta_{m}^{-1}.$$
This means $D_\infty[2^m+1,2^{m+1}-1]=E_{m,1}$.

Consider the first $A_{m-1}$ in $E_{m,1}=A_m\delta_m^{-1}=\underline{A_{m-1}}B_{m-1}\delta_m^{-1}$. By Lemma 2.5(1), it
occurs in $D_\infty[1,3\ast2^{m}]=A_{m-1}B_{m-1}A_{m-1}$ twice at positions 1 and $2^{m}+1$.
When the $A_{m-1}$ occurs at positions 1, $E_{m,1}=D_\infty[1,2^m-1]$.
By the proof of Theorem 4.1(1), it is $E_{m,1,1}$, i.e,
$$E_{m,1,1}=D_\infty[1,2^m-1].$$

When the $A_{m-1}$ occurs at positions $2^{m}+1$, $E_{m,1}=D_\infty[2^m+1,2^{m+1}-1]$, it is $E_{m,1,2}$, i.e,
$$E_{m,1,2}=D_\infty[2^m+1,2^{m+1}-1].$$

Using similar arguments and by Lemma 2.5(2), we can determine the positions of $E_{m,1,3}$,
$$E_{m,1,3}=D_\infty[3\ast2^{m-1}+1,5\ast2^{m-1}-1].$$

This means:

$G_1(E_{m,1})=D_\infty[2^m]=\delta_m$;

$G_2(E_{m,1})=D_\infty[3\ast2^{m-1}+1,2^{m+1}-1]^{-1}=\delta_{m-1}A_{m-1}^{-1}$.
\end{proof}

\begin{Theorem}[The first four gaps of $E_{m,2}$] For $m\geq1$,

(1) $G_1(E_{m,2})=G_3(E_{m,2})=\delta_{m}A_{m}^{-1}$;

(2) $G_2(E_{m,2})=\delta_mB_{m+1}$;

(3) $G_4(E_{m,2})=\delta_m$.
\end{Theorem}

\begin{proof} By the definition of envelope word, $E_{m,2}=A_{m-1}A_m\delta_m^{-1}$. Using similar argument, we determine the positions of $E_{m,2,i}$, $i=1,2,3,4,5$ as follows:

$E_{m,2,1}=D_\infty[2\ast2^{m-1}+1,5\ast2^{m-1}-1]$, (by proof of Theorem 4.1(2));

$E_{m,2,2}=D_\infty[3\ast2^{m-1}+1,6\ast2^{m-1}-1]$, (by Lemma 2.5(2));

$E_{m,2,3}=D_\infty[10\ast2^{m-1}+1,13\ast2^{m-1}-1]$, (by Lemma 2.5(3));

$E_{m,2,4}=D_\infty[11\ast2^{m-1}+1,14\ast2^{m-1}-1]$, (by Lemma 2.5(2));

$E_{m,2,5}=D_\infty[14\ast2^{m-1}+1,17\ast2^{m-1}-1]$, (by Lemma 2.5(4)).

This means:

$G_1(E_{m,2})=D_\infty[5\ast2^{m-1}-1,3\ast2^{m-1}+1]^{-1}=\delta_{m}A_{m}^{-1}$;

$G_2(E_{m,2})=D_\infty[6\ast2^{m-1},10\ast2^{m-1}]=\delta_mB_{m+1}$;

$G_3(E_{m,2})=D_\infty[11\ast2^{m-1}+1,13\ast2^{m-1}-1]=\delta_{m}A_{m}^{-1}$;

$G_4(E_{m,2})=D_\infty[14\ast2^{m-1}]=\delta_m$.
\end{proof}

\begin{definition}[Sequence $\Theta_i$, $i=1,2$]\

(1) Let substitution $\varphi_1(a,b)=(a,bb)$, then $\Theta_1=\varphi_1(D_\infty)$;

(2) Let substitution $\varphi_2(a,b)=(ab,acac)$, then $\Theta_2=\varphi_2(D_\infty)$.
\end{definition}

\noindent\textbf{Remark.}

(1) $\Theta_1=abbaaabbabbabbaaabbaaabbaaabbabbabbaaabbabbabbaaabbabbabbaaabbaaabbaaa\cdots$;

(2) $\Theta_2=abacacabababacacabacacabacacabababacacabababacacabababacacabacacabacac\cdots$.

\begin{Theorem}
The gap sequence of $E_{m,1}$ is sequence $\Theta_1$.
\end{Theorem}

\begin{proof} We write $G_1(E_{m,1})$ (resp. $G_2(E_{m,1})$) as $G_1$ (resp. $G_2$) for short.

\textbf{Step 1.} Since $E_{m,1}=A_m\delta_m^{-1}=B_m\delta_{m-1}^{-1}$, $G_1=\delta_m$ and $G_2=\delta_{m-1}A_{m-1}^{-1}$.

(a) $E_{m,1}G_1=A_m\delta_m^{-1}\ast\delta_m=A_m$;

(b) $E_{m,1}G_2E_{m,1}G_2=B_m\delta_{m-1}^{-1}\ast\delta_{m-1}A_{m-1}^{-1}
\ast B_m\delta_{m-1}^{-1}\ast\delta_{m-1}A_{m-1}^{-1}=B_m$.

\textbf{Step 2.} By the definition of sequence $\Theta_1$, $\sigma(a,bb)=(abb,aa)$. So we only need to prove:

(1) $\sigma(E_{m,1}G_1)=E_{m,1}G_1E_{m,1}G_2E_{m,1}G_2$;

(2) $\sigma(E_{m,1}G_2E_{m,1}G_2)=E_{m,1}G_1E_{m,1}G_1$.

In fact,

(1) $\sigma(E_{m,1}G_1)=\sigma(A_m\delta_m^{-1}\delta_m)=A_{m+1}
=A_mB_m=E_{m,1}G_1E_{m,1}G_2E_{m,1}G_2$;

(2) $\sigma(E_{m,1}G_2E_{m,1}G_2)=\sigma(B_m)=B_{m+1}=A_mA_m=E_{m,1}G_1E_{m,1}G_1$.

So the theorem holds.
\end{proof}

\noindent\textbf{Example.} Consider $E_{2,1}=aba$, then:
\begin{equation*}
\begin{split}
D_\infty=&(aba)\underbrace{a}_A(ab\underbrace{(a)}_Bb\underbrace{(a)}_Bba)\underbrace{a}_A
(aba)\underbrace{a}_A(aba)\underbrace{a}_A(ab\underbrace{(a)}_Bb\underbrace{(a)}_B
ba)\underbrace{a}_A\\
&(ab\underbrace{(a)}_Bb\underbrace{(a)}_Bba)\underbrace{a}_A(ab\underbrace{(a)}_B
b\underbrace{(a)}_Bba)\underbrace{a}_A(aba)\underbrace{a}_A(aba)\underbrace{a}_A
(ab\underbrace{(a)}_Bb\underbrace{(a)}_Bba)\\
&\underbrace{a}_A(aba)\underbrace{a}_A(aba)\underbrace{a}_A
(ab\underbrace{(a)}_Bb\underbrace{(a)}_Bba)
\underbrace{a}_A(aba)\underbrace{a}_A(aba)\underbrace{a}_A\cdots
\end{split}
\end{equation*}

\begin{Theorem}
The gap sequence of $E_{m,2}$ is sequence $\Theta_2$.
\end{Theorem}

\begin{proof} We write $G_1(E_{m,2})$ (resp. $G_2(E_{m,2})$, $G_4(E_{m,2})$) as $G_1$ (resp. $G_2$, $G_4$) for short.

\textbf{Step 1.} $E_{m,2}=A_{m-1}A_m\delta_m^{-1}$, $G_1=\delta_{m}A_{m}^{-1}$, $G_2=\delta_mB_{m+1}$ and $G_4=\delta_m$.

(a) $E_{m,2}G_1E_{m,2}G_2=A_{m-1}A_m\delta_m^{-1}\ast \delta_{m}A_{m}^{-1}\ast A_{m-1}A_m\delta_m^{-1}\ast \delta_mB_{m+1}=B_mA_mB_{m+1}$;

(b) $E_{m,2}G_1E_{m,2}G_4=A_{m-1}A_m\delta_m^{-1}\ast \delta_{m}A_{m}^{-1}\ast A_{m-1}A_m\delta_m^{-1}\ast \delta_m=B_mA_m$;

\textbf{Step 2.} By the definition of sequence $\Theta_2$, $\sigma(ab,acac)=(abacac,abab)$. So we only need to prove the two properties below, where $\eta=B_m$ is a shift.

(1) $\eta\sigma(E_{m,2}G_1E_{m,2}G_2)\eta^{-1}=E_{m,2}G_1E_{m,2}G_2E_{m,2}G_1E_{m,2}G_4E_{m,2}G_1E_{m,2}G_4$;

(2) $\eta\sigma(E_{m,2}G_1E_{m,2}G_4E_{m,2}G_1E_{m,2}G_4)\eta^{-1}=E_{m,2}G_1E_{m,2}G_2E_{m,2}G_1E_{m,2}G_2$.

In fact,
\begin{equation*}
\begin{split}
&\eta\sigma(E_{m,2}G_1E_{m,2}G_2)\eta^{-1}=\eta\sigma(B_mA_mB_{m+1})\eta^{-1}=B_m\ast A_mA_mA_mB_mA_mB_mA_mB_m\ast B_m^{-1}\\
=&B_mA_mA_mA_mB_mA_mB_mA_m=E_{m,2}G_1E_{m,2}G_2E_{m,2}G_1E_{m,2}G_4E_{m,2}G_1E_{m,2}G_4;\\
&\eta\sigma(E_{m,2}G_1E_{m,2}G_4E_{m,2}G_1E_{m,2}G_4)\eta^{-1}
=\eta\sigma(B_mA_mB_mA_m)\eta^{-1}\\
=&B_m\ast A_mA_mA_mB_mA_mA_mA_mB_m\ast B_m^{-1}\\
=&B_mA_mA_mA_mB_mA_mA_mA_m=E_{m,2}G_1E_{m,2}G_2E_{m,2}G_1E_{m,2}G_2.
\end{split}
\end{equation*}

\textbf{Step 3.} We must prove the shift $\eta$ is allowable. In fact, $G_0(E_{m,i})=A_m$, then
$$\sigma(G_0(E_{m,i}))=\sigma(A_m)=A_mB_m=G_0(E_{m,i})\ast\eta.$$
So under the substitution $\sigma$, there is a shift $\eta$.
\end{proof}

\noindent\textbf{Example.} Consider $E_{1,2}=aa$, then:
\begin{equation*}
\begin{split}
D_\infty=&\underbrace{ab}_{G_0(aa)}(a\underbrace{(a)}_Aa)\underbrace{babab}_B
(a\underbrace{(a)}_Aa)\underbrace{b}_C(a\underbrace{(a)}_Aa)\underbrace{b}_C
(a\underbrace{(a)}_Aa)\underbrace{babab}_B(a\underbrace{(a)}_Aa)\underbrace{babab}_B\\
&(a\underbrace{(a)}_Aa)\underbrace{babab}_B(a\underbrace{(a)}_Aa)\underbrace{b}_C
(a\underbrace{(a)}_Aa)\underbrace{b}_C(a\underbrace{(a)}_Aa)\underbrace{babab}_B
(a\underbrace{(a)}_Aa)\underbrace{b}_C(a\underbrace{(a)}_Aa)\\
&\underbrace{b}_C(a\underbrace{(a)}_Aa)\underbrace{babab}_B(a\underbrace{(a)}_Aa)
\underbrace{b}_C(a\underbrace{(a)}_Aa)\underbrace{b}_C(a\underbrace{(a)}_Aa)\underbrace{babab}_B
\cdots
\end{split}
\end{equation*}


\vspace{0.5cm}

\stepcounter{section}

\noindent\textbf{\large{5.~Gaps and Gap Sequence of an Arbitrary Word}}

\vspace{0.4cm}

In Section 3, we give the weak version of ``uniqueness of envelope extension property".
In this section, we are going to give the strong version. These properties make ``envelope word" quite special. Using them, we can extend the properties about gaps and gap sequence from envelope word to an arbitrary word in Doubling sequence $D_\infty$.

\begin{Theorem}[Uniqueness of envelope extension, strong version]\

Let $Env(\omega)=E_{m,i}$, $i=1,2$, then $\omega_p$ has an unique envelope extension as
$$E_{m,i,p}=\mu_1(\omega)\ast \omega_p\ast\mu_2(\omega),$$
where $\mu_1(\omega)$ and $\mu_2(\omega)$ are constant words depending only on $\omega$.
\end{Theorem}

\begin{proof} The key is to prove that $Env(\omega_p)=Env(\omega)_p$, i.e. $Env(\omega_p)=E_{m,i,p}$.

We write $G_j(Env(\omega))$ as $G_j$ for short. We only need to prove that:

(a) $\omega$ occurs only once in $G_0Env(\omega)$;

(b) $\omega$ occurs only twice in $Env(\omega)G_jEnv(\omega)$.

\textbf{Part 1.} When $i=1$, $E_{m,1}=A_m\delta_m^{-1}$.
$G_0=\varepsilon$, $G_1=\delta_m$, $G_2=\delta_{m-1}A_{m-1}^{-1}$.
By the proof of Theorem 3.7, we have $\delta_{m-3}A_{m-3}\prec\omega$.

In this part, We write $A_{m-3}$ (resp. $B_{m-3}$) as $A$ (resp. $B$) for short.

\textbf{(1)} $G_0E_{m,1}=E_{m,1}$, so $\omega$ occurs only once in $G_0E_{m,1}$.

\textbf{(2)} $E_{m,1}G_1E_{m,1}=A_m\delta_m^{-1}\delta_mA_m\delta_m^{-1}=A_mA_m\delta_m^{-1}$.
\setlength{\unitlength}{1mm}
\begin{center}
\begin{picture}(90,25)
\put(0,10){$E_{m,1,1}G_1E_{m,1,2}$}
\put(27,10){$=$}
\put(30,10){$A$}
\put(33,10){$B$}
\put(36,10){$A$}
\put(39,10){$A$}
\put(42,10){$A$}
\put(45,10){$B$}
\put(48,10){$A$}
\put(51,10){$B$}
\put(54,10){$A$}
\put(57,10){$B$}
\put(60,10){$A$}
\put(63,10){$A$}
\put(66,10){$A$}
\put(69,10){$B$}
\put(72,10){$A$}
\put(75,10){$B$}
\put(78,10){$\delta_m^{-1}$}
\put(40,5){\vector(0,1){4}}
\put(43,18){\vector(0,-1){4}}
\put(64,5){\vector(0,1){4}}
\put(67,18){\vector(0,-1){4}}
\put(39,1){[1]=$3\ast2^{m-3}$}
\put(42,20){[2]=$2^{m-1}$}
\put(63,1){[3]=$3\ast2^{m-1}$}
\put(66,20){[4]=$13\ast2^{m-3}$}
\end{picture}
\end{center}
\centerline{Fig. 5.1: The 4 possible positions of $\delta_{m-3}A_{m-3}$ in $E_{m,1,1}G_1E_{m,1,2}$.}

\vspace{0.2cm}

(a) When $\delta_{m-3}A_{m-3}$ occur at positions [1] and [2], $\omega\prec E_{m,1,1}$;

(b) When $\delta_{m-3}A_{m-3}$ occur at positions [3] and [4], $\omega\prec E_{m,1,2}$.

Since $\omega$ occurs once in $E_{m,1}$, $\omega$ occurs twice in $E_{m,1}G_1E_{m,1}$.

\textbf{(3)} $E_{m,1}G_2E_{m,1}=A_m\delta_m^{-1}\delta_{m-1}A_{m-1}^{-1}A_m\delta_m^{-1}
=A_{m-1}A_m\delta_m^{-1}$.
\setlength{\unitlength}{1mm}
\begin{center}
\begin{picture}(80,25)
\put(0,10){$E_{m,1,2}G_2E_{m,1,3}$}
\put(27,10){$=$}
\put(30,10){$A$}
\put(33,10){$B$}
\put(36,10){$A$}
\put(39,10){$A$}
\put(42,10){$A$}
\put(45,10){$B$}
\put(48,10){$A$}
\put(51,10){$A$}
\put(54,10){$A$}
\put(57,10){$B$}
\put(60,10){$A$}
\put(63,10){$B$}
\put(66,10){$\delta_m^{-1}$}
\put(40,5){\vector(0,1){4}}
\put(43,18){\vector(0,-1){4}}
\put(52,5){\vector(0,1){4}}
\put(55,18){\vector(0,-1){4}}
\put(26,1){[1]=$3\ast2^{m-3}$}
\put(34,20){[2]=$2^{m-1}$}
\put(51,1){[3]=$7\ast2^{m-3}$}
\put(54,20){[4]=$2^{m}$}
\end{picture}
\end{center}
\centerline{Fig. 5.2: The 4 possible positions of $\delta_{m-3}A_{m-3}$ in $E_{m,1,2}G_2E_{m,1,3}$.}

\vspace{0.2cm}

(a) When $\delta_{m-3}A_{m-3}$ occur at positions [1] and [2], $\omega\prec E_{m,1,2}$;

(b) When $\delta_{m-3}A_{m-3}$ occur at positions [3] and [4], $\omega\prec E_{m,1,3}$.

Since $\omega$ occurs once in $E_{m,1}$, $\omega$ occurs twice in $E_{m,1}G_2E_{m,1}$.

\vspace{0.2cm}

\textbf{Part 2.} When $i=2$, $E_{m,2}=A_{m-1}A_m\delta_m^{-1}$.
$$G_0=A_m,~G_1=\delta_{m}A_{m}^{-1},~G_2=\delta_mB_{m+1},~G_4=\delta_m.$$
By the proof of Theorem 3.7, we have $\delta_{m-1}A_{m-1}\prec\omega$.

In this part, We denote $A_{m-1}$ (resp. $B_{m-1}$) as $A$ (resp. $B$) for short.

\textbf{(1)} $G_0E_{m,2}=A_mA_{m-1}A_m\delta_m^{-1}=ABA\underline{A}B\delta_m^{-1}$.

The possible position of $\delta_{m-1}A_{m-1}$ in $G_0E_{m,2}$ is $3\ast2^{m-1}$.

\textbf{(2)} $E_{m,2}G_1E_{m,2}=A_{m-1}A_m\delta_m^{-1}\delta_{m}A_{m}^{-1}A_{m-1}A_m\delta_m^{-1}
=B_{m}A_m\delta_m^{-1}=A\underline{A}\ast\underline{A}B\delta_m^{-1}$.

The possible positions of $\delta_{m-1}A_{m-1}$ in $E_{m,2}G_1E_{m,2}$ are $2^{m-1}$ and $2^m$.

\textbf{(3)} $E_{m,2}G_2E_{m,2}=A_{m-1}A_m\delta_m^{-1}\delta_mB_{m+1}A_{m-1}A_m\delta_m^{-1}
=A\underline{A}BABABA\underline{A}B$.

The possible positions of $\delta_{m-1}A_{m-1}$ in $E_{m,2}G_2E_{m,2}$ are
$2^{m-1}$ and $2^{m+3}$.

\textbf{(4)} $E_{m,2}G_4E_{m,2}=A_{m-1}A_m\delta_m^{-1}\delta_mA_{m-1}A_m\delta_m^{-1}
=A\underline{A}BA\underline{A}B$.

The possible positions of $\delta_{m-1}A_{m-1}$ in $E_{m,2}G_4E_{m,2}$ are
$2^{m-1}$ and $2^{m+2}$.

Since $\omega$ occurs once in $E_{m,2}$,
$\omega$ occurs once in $G_0E_{m,2}$,
twice in $E_{m,2}G_iE_{m,2}$, $i=1,2,4$.
\end{proof}

\begin{Theorem}[Gap sequence of an arbitrary word $\omega$]
Let $Env(\omega)=E_{m,i}$, $i=1,2$, then

(1) When $i=1$, the gap sequence of $\omega$ is sequence $\Theta_1$ on alphabet $\{G_1(\omega),G_2(\omega)\}$;

(2) When $i=2$, the gap sequence of $\omega$ is sequence $\Theta_2$ on alphabet $\{G_1(\omega),G_2(\omega),G_4(\omega)\}$.
\end{Theorem}

\begin{proposition}[Relation between $G_p(\omega)$ and $G_p(Env(\omega))$, $p\geq1$]\

Let $Env(\omega)=E_{m,i}$, $i=1,2$, then
$$G_p(\omega)=\mu_2(\omega)\ast G_p(E_{m,i})\ast\mu_1(\omega),~p\geq1.$$
\end{proposition}

\begin{proof} The proof of the proposition is easy by the following diagram.
\setlength{\unitlength}{1mm}
\begin{center}
\begin{picture}(120,15)
\linethickness{1pt}
\put(0,5){\line(1,0){15}}
\linethickness{3pt}
\put(15,5){\line(1,0){20}}
\linethickness{1pt}
\put(35,5){\line(1,0){55}}
\linethickness{3pt}
\put(90,5){\line(1,0){20}}
\linethickness{1pt}
\put(110,5){\line(1,0){10}}
\put(0,5){\line(0,1){10}}
\put(45,5){\line(0,1){10}}
\put(75,5){\line(0,1){10}}
\put(120,5){\line(0,1){10}}
\put(15,0){\line(0,1){10}}
\put(35,0){\line(0,1){10}}
\put(90,0){\line(0,1){10}}
\put(110,0){\line(0,1){10}}
\put(6,7){$\mu_1$}
\put(81,7){$\mu_1$}
\put(38,7){$\mu_2$}
\put(113,7){$\mu_2$}
\put(23,0){$\omega_p$}
\put(96,0){$\omega_{p+1}$}
\put(19,12){$E_{m,i,p}$}
\put(92,12){$E_{m,i,p+1}$}
\put(52,9){$G_p(E_{m,i})$}
\put(55,0){$G_p(\omega)$}
\put(18,13){\vector(-1,0){18}}
\put(30,13){\vector(1,0){15}}
\put(91,13){\vector(-1,0){16}}
\put(106,13){\vector(1,0){14}}
\put(51,10){\vector(-1,0){6}}
\put(69,10){\vector(1,0){6}}
\put(53,1){\vector(-1,0){18}}
\put(68,1){\vector(1,0){22}}
\end{picture}
\end{center}
\centerline{Fig. 5.3: The relation among $\omega_p$, $E_{m,i,p}$ $G_p(\omega)$ and $G_p(E_{m,i})$.}
\end{proof}

\begin{proposition}[Relation between $G_0(\omega)$ and $G_0(Env(\omega))$]\

Let $Env(\omega)=E_{m,i}$, $i=1,2$, then
$$G_0(\omega)=G_0(E_{m,i})\ast\mu_1(\omega).$$
\end{proposition}

\begin{proof} The proof of the proposition is easy by the following diagram.
\setlength{\unitlength}{1mm}
\begin{center}
\begin{picture}(75,15)
\linethickness{1pt}
\put(0,5){\line(1,0){45}}
\linethickness{3pt}
\put(45,5){\line(1,0){20}}
\linethickness{1pt}
\put(65,5){\line(1,0){10}}
\put(30,5){\line(0,1){10}}
\put(75,5){\line(0,1){10}}
\put(0,0){\line(0,1){15}}
\put(45,0){\line(0,1){10}}
\put(65,0){\line(0,1){10}}
\put(36,7){$\mu_1$}
\put(68,7){$\mu_2$}
\put(53,0){$\omega_1$}
\put(50,12){$E_{m,i,1}$}
\put(8,9){$G_0(E_{m,i})$}
\put(17,0){$G_0(\omega)$}
\put(49,13){\vector(-1,0){19}}
\put(61,13){\vector(1,0){14}}
\put(7,10){\vector(-1,0){7}}
\put(24,10){\vector(1,0){6}}
\put(15,1){\vector(-1,0){15}}
\put(30,1){\vector(1,0){15}}
\end{picture}
\end{center}
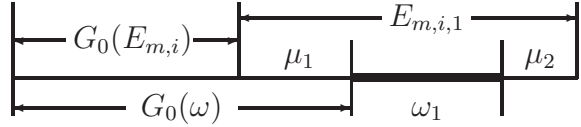
\centerline{Fig. 5.4: The relation among $\omega_1$, $E_{m,i,1}$ $G_0(\omega)$ and $G_0(E_{m,i})$.}
\end{proof}

Using Theorem 5.1, Proposition 5.3 and 5.4, we can give the expressions of $G_p(\omega)$.
The proofs are simple, while the expressions are complicated, we prefer to omit them.


\vspace{0.5cm}

\stepcounter{section}

\noindent\textbf{\large{6.~The Positions of $\omega_p$}}

\vspace{0.4cm}

The position of $\omega_p$ and $\omega_q~(p\neq q)$ are distinct.
In fact, $\omega_p$ should be regarded as two variables $\omega$ and $p$, where $\omega\prec D_\infty$ and $p\geq1$.
In this section, we will give the position of $\omega_p$ for all $(\omega,p)$.

\begin{definition} Let $N_j(x,p):=|\Theta_j[1,p]|_x$ be the number of letter $x$ occurring in $\Theta_j[1,p]$, where $j=1,2$ and $p\geq0$. When $j=1$,
letter $x$ is an element of alphabet $\{a,b\}$; when $j=2$, letter $x$ is an element of alphabet $\{a,b,c\}$. \end{definition}

\noindent\textbf{Example.} Since $\Theta_1=abbaaabbabbab\cdots$, $N_1(a,5)=|\Theta_1[1,5]|_a=3$, $N_1(b,9)=|\Theta_1[1,9]|_b=4$.

\begin{proposition}[Positions of $E_{m,i,p}$]\

(1) $L(E_{m,1},p+1)=p\ast2^m-N_1(b,p)\ast2^{m-1}+1$;

(2) $L(E_{m,2},p+1)=(3p+2)\ast2^{m-1}-N_2(a,p)\ast2^m+N_2(b,p)\ast2^{m+1}+1$.
\end{proposition}

\begin{proof} In sequence $D_\infty$, there are $p$ ``$E_{m,i}"$, one ``$G_0"$ and gaps $G_j(\omega)$ ($1\leq j\leq p$) ahead.

When $i=1$, $E_{m,1}=A_m\delta_m^{-1}$,
$G_0=\varepsilon$, $G_1=\delta_m$, $G_2=\delta_{m-1}A_{m-1}^{-1}$.
\begin{equation*}
\begin{split}
&L(E_{m,1},p+1)-1
=p\ast|E_{m,1}|+|G_0|+N_1(a,p)\ast|G_1|+N_1(b,p)\ast|G_2|\\
=&p\ast|A_m\delta_m^{-1}|+|\varepsilon|+N_1(a,p)\ast|\delta_m|+N_1(b,p)\ast|\delta_{m-1}A_{m-1}^{-1}|\\
=&p\ast2^m-p+N_1(a,p)+N_1(b,p)-N_1(b,p)\ast2^{m-1}
=p\ast2^m-N_1(b,p)\ast2^{m-1}.
\end{split}
\end{equation*}

When $i=2$, $E_{m,2}=A_{m-1}A_m\delta_m^{-1}$,
$G_0=A_m$, $G_1=\delta_{m}A_{m}^{-1}$, $G_2=\delta_mB_{m+1}$, $G_4=\delta_m$.
\begin{equation*}
\begin{split}
&L(E_{m,2},p+1)-1
=p\ast|E_{m,2}|+|G_0|+N_2(a,p)\ast|G_1|+N_2(b,p)\ast|G_2|+N_2(c,p)\ast|G_4|\\
=&p\ast|A_{m-1}A_m\delta_m^{-1}|+|A_m|+N_2(a,p)\ast|\delta_{m}A_{m}^{-1}|
+N_2(b,p)\ast|\delta_mB_{m+1}|+N_2(c,p)\ast|\delta_m|\\
=&3p\ast2^{m-1}-p+2^m+N_2(a,p)-N_2(a,p)\ast2^m+N_2(b,p)+N_2(b,p)\ast2^{m+1}+N_2(c,p)\\
=&(3p+2)\ast2^{m-1}-N_2(a,p)\ast2^m+N_2(b,p)\ast2^{m+1}.
\end{split}
\end{equation*}
\end{proof}

\begin{proposition}[Position of $\omega_p$]\

Let $Env(\omega)=E_{m,i}$ and $Env(\omega)=\mu_1\ast\omega\ast\mu_2$, then
$L(\omega,p)=L(E_{m,i},p)+|\mu_1|$, $i=1,2$.
\end{proposition}


\vspace{0.5cm}

\stepcounter{section}

\noindent\textbf{\large{7.~Combinatorial Properties of Factors}}

\vspace{0.4cm}

As an application, we will give some combinatorial properties of factors in sequence $D_\infty$.
Damanik \cite{D2000} gave a description of the sets of palindromes and powers occurring in Doubling sequence.
Using the properties of gaps and gap sequence of envelope words in Doubling sequence, we can get more explicit results.

\vspace{0.4cm}

\noindent\emph{7.1~Palindrome}

\vspace{0.4cm}

Since envelope words $E_{m,i}$ are palindromes, we can determine all palindromes with envelope $Env(\omega)=E_{m,i}$. Using it, we can give a description of the set of  palindromes.

\begin{property} Let $\omega$ be a palindrome, $Env(\omega)=E_{m,i}$ and $E_{m,i}=\mu_1\omega\mu_2$, then $|\mu_1|=|\mu_2|$.

(1) When $i=1$ and $m=1$ (resp. $2$), $|\mu_1|=0$ (resp. $|\mu_1|=0,1$);

(2) When $i=1$ and $m\geq3$, $0\leq|\mu_1|<3\ast2^{m-3}$;

(3) When $i=2$ and $m\geq1$, $0\leq|\mu_1|<2^{m-1}$.
\end{property}

\begin{proof} By Property 3.4, envelope words are palindromes.
Suppose $|\mu_1|<|\mu_2|$, then
$$\overleftarrow{E_{m,i}}=\overleftarrow{\mu_1\omega\mu_2}
=\overleftarrow{\mu_2}\overleftarrow{\omega}\overleftarrow{\mu_1}
=\overleftarrow{\mu_2}\omega\overleftarrow{\mu_1}=E_{m,i}=\mu_1\omega\mu_2.$$
This means $\omega$ occurs in $E_{m,i}$ twice. It contradicts Theorem 3.7.
So $|\mu_1|=|\mu_2|$.

Consider the range of $|\mu_1|$.

\textbf{Case 1}: $E_{m,1}=A_m\delta_m^{-1}$.
For $m\geq3$, suppose $|\mu_1|=3\ast2^{m-3}+h$ where $h\geq0$, then
\begin{equation*}
\begin{split}
&E_{m,1}[3\ast2^{m-3}+h+1,5\ast2^{m-3}-h-1]\\
=&(A_{m-3}B_{m-3}A_{m-3}A_{m-3}A_{m-3}B_{m-3}A_{m-3}B_{m-3})[3\ast2^{m-3}+h+1,5\ast2^{m-3}-h-1]\\
=&(A_{m-3}A_{m-3})[h+1,2^{m-2}-h-1]\prec A_{m-2}\delta_{m-2}^{-1}=E_{m-2,1}\sqsubset E_{m,1}.
\end{split}
\end{equation*}

When $|\mu_1|<3\ast2^{m-3}$, $A_{m-3}A_{m-3}\prec\omega$. Since $A_{m-3}A_{m-3}$ isn't the factor of envelope word which less than $E_{m,1}$, then $0\leq|\mu_1|<3\ast2^{m-3}$ in this case.

When $m=1$, $|\mu_1|=0$; when $m=2$, $|\mu_1|=0$ or $1$. 

\textbf{Case 2}: $E_{m,2}=A_{m-1}A_m\delta_m^{-1}$.
For $m\geq1$, suppose $|\mu_1|=2^{m-1}+h$ where $h\geq0$, then
\begin{equation*}
\begin{split}
&E_{m,2}[2^{m-1}+h+1,2^{m}-h-1]
=(A_{m-1}A_{m-1}B_{m-1})[2^{m-1}+h+1,2^{m}-h-1]\\
=&A_{m-1}[h+1,2^{m-1}-h-1]\prec A_{m-1}\delta_{m-1}^{-1}=E_{m-1,1}\sqsubset E_{m,2}.
\end{split}
\end{equation*}

When $|\mu_1|<2^{m-1}$, $A_{m-1}\prec\omega$. Since $A_{m-1}$ isn't the factor of envelope word which is less than $E_{m,2}$, $0\leq|\mu_1|<2^{m-1}$ in this case.
\end{proof}

\begin{corollary} All palindromes in sequence $D_\infty$ are of odd length, except ``$aa$". \end{corollary}

\begin{proof} Let $\omega$ be a palindrome. Since all envelope words are of odd length, except $E_{1,2}=aa$, and
$|\omega|=|Env(\omega)|-2\ast|\mu_1|,$
$|\omega|$ is odd.
\end{proof}

\begin{proposition} Let $P(n)$ be the number of palindromes of length $n$, then
$P(1)=2$, $P(2)=1$, $P(3)=3$. For $n\geq3$, $P(2n-2)=0$ and
\begin{equation*}
P(2n-1)=\begin{cases}
3,&3\ast2^{n}+1\leq n\leq2^{n-2};\\
4,&2^{n-2}+1\leq n\leq 3\ast2^{n+1}.
\end{cases}
\end{equation*}
\end{proposition}

\begin{proof} All palindromes with envelope $E_{1,1}$ or $E_{2,1}$ are $\{a,b,aba\}$.

All lengths of palindromes with envelope $E_{m+1,1}$ for $m\geq2$ are $$\{2^{m-1}+1,2^{m-1}+3,\ldots,2^{m+1}-3,2^{m+1}-1\}.$$
There are one palindrome of each length above.

All lengths of palindromes with envelope $E_{m,2}$ for $m\geq1$ are $$\{2^{m-1}+1,2^{m-1}+3,\ldots,3\ast2^{m-1}-3,3\ast2^{m-1}-1\}.$$
There are one palindrome of each length above too.
So the proposition holds.
\end{proof}

\vspace{0.4cm}

\noindent\emph{7.2~Power, Overlap and Separate Properties Between $\omega_p$ and $\omega_{p-1}$}


\begin{property} Let $Env(\omega)=E_{m,1}$, then $G_0(\omega)\geq0$, $G_1(\omega)>0$ and
\begin{equation*}
G_2(\omega)\begin{cases}
>0,&|\omega|<2^{m-1},\\
=0,&|\omega|=2^{m-1},\\
<0,&|\omega|>2^{m-1}.
\end{cases}
\end{equation*}
\end{property}

\begin{property} Let $Env(\omega)=E_{m,2}$, then $G_0(\omega)>0$, $G_1(\omega)<0$,
$G_2(\omega)>0$, $G_4(\omega)>0$.
\end{property}

By Theorem 4.2, 4.3 and Proposition 5.3, all factor $\omega\in D_\infty$ can be divided into four types according to the different lengths of gaps. We denote those types by $T_j$, $j=1,2,3,4$.
Obviously, the disjoint union of sets $T_j$, $j=1,2,3,4$, consists of all factors in sequence $D_\infty$.

\begin{definition}[Types] The sets $T_j$ are defined as follow:

$T_1=\{\omega\in D_\infty: Env(\omega)=E_{m,1},~|\omega|<2^{m-1}\}$;

$T_2=\{\omega\in D_\infty: Env(\omega)=E_{m,1},~|\omega|=2^{m-1}\}$;

$T_3=\{\omega\in D_\infty: Env(\omega)=E_{m,1},~|\omega|>2^{m-1}\}$;

$T_4=\{\omega\in D_\infty: Env(\omega)=E_{m,2}\}$.
\end{definition}

\begin{corollary}\

(1) $\omega\in T_1\iff G_0(\omega)\geq0,~G_1(\omega)>0,~G_2(\omega)>0$;

(2) $\omega\in T_2\iff G_0(\omega)\geq0,~G_1(\omega)>0,~G_2(\omega)=0$;

(3) $\omega\in T_3\iff G_0(\omega)\geq0,~G_1(\omega)>0,~G_2(\omega)<0$;

(4) $\omega\in T_4\iff G_0(\omega)>0,~G_1(\omega)<0,~G_2(\omega)>0,~G_4(\omega)>0$.
\end{corollary}

We will study some combinatorial properties such as adjacent, separated and overlapped of factors. From our knowledge, all previous studies on combinatorial only consider factor $\omega$. We consider not only $\omega$ but also $p$, which is much more difficult.

\begin{definition}[Properties] Let $\omega\in D_\infty$, $p\geq1$, $i\geq1$, define

$\mathcal{P}_i=\{(\omega,p):G_j(\omega)=0,j=p,p+1,\ldots,p+i-1\}
=\{(\omega,p):\omega_p\cdots\omega_{p+i}\prec D_\infty\}$;

$\mathcal{S}_i=\{(\omega,p):G_j(\omega)>0,j=p,p+1,\ldots,p+i-1\}$, $\mathcal{S}_\infty=\{(\omega,p):G_j(\omega)>0,\forall j\geq1\}$;

$\mathcal{O}_i=\{(\omega,p):G_j(\omega)<0,j=p,p+1,\ldots,p+i-1\}$.
\end{definition}

\begin{property}[Power Property]
$\mathcal{P}_1=(T2,\Gamma_1(b))$,
$\mathcal{P}_2=(T2,\Gamma_1(bb))$,
$\mathcal{P}_j=\emptyset$ for $j\geq3$.
\end{property}

\begin{property}[Separated Property]\

$\mathcal{S}_1=(T1,\mathbb{N})\cup(T2\cup T3,\Gamma_1(a))\cup(T4,\Gamma_2(b)\cup\Gamma_2(c))$;

$\mathcal{S}_2=(T1,\mathbb{N})\cup(T2\cup T3,\Gamma_1(aa))$;

$\mathcal{S}_3=(T1,\mathbb{N})\cup(T2\cup T3,\Gamma_1(aaa))$;

$\mathcal{S}_j=(T1,\mathbb{N})$ for $j\geq4$ and $j=\infty$.
\end{property}

\begin{property}[Overlapped Property]\

$\mathcal{O}_1=(T3,\Gamma_1(b))\cup(T4,\Gamma_2(a))$,
$\mathcal{O}_2=(T3,\Gamma_1(bb))$,
$\mathcal{O}_j=\emptyset$ for $j\geq3$.
\end{property}

\vspace{0.4cm}

\noindent\emph{7.3~Squares and Cubes in Sequence $D_\infty$}

\vspace{0.4cm}

Brown, Rampersad, Shallit and Vasiga \cite{BRSV2006} described all the positions of squares occurring in Thue-Morse sequence $M_\infty$. They also counted the number of distinct squares beginning in $M_\infty[1,N]$ for $N\geq1$.
In this subsection, we are going to
count the number of distinct squares (resp. cubes) beginning in $D_\infty[1,N]$ for $N\geq1$.
Factor $\omega_p$ is said to be begin in $D_\infty[1,N]$, if the first letter of $\omega_p$ occurs in $D_\infty[1,N]$.

\begin{property}[Number of distinct squares and cubes]\

Let $c(N)$ be the number of distinct squares beginning in $D_\infty[1,N]$. Then
\begin{equation*}
c(N+1)-c(N)=\begin{cases}
1,&2^m\leq N<3\ast2^{m-1};\\
0,&3\ast2^{m-1}\leq N<2^{m+1}.
\end{cases}
\end{equation*}
where $m\geq1$.
The number of distinct cubes beginning in $D_\infty[1,N]$ is $c(N)$ too.
\end{property}

\begin{proof} Let $\omega\omega\prec D_\infty$, then $\omega\in T2$. Let $Env(\omega)=E_{m,1}$ and $E_{m,1}=\mu_1\omega\mu_2$. By Property 7.9, $\mathcal{P}_1=(T2,\Gamma_1(b))$ and
$\mathcal{P}_2=(T2,\Gamma_1(bb))$.
Since $\Theta_1[2]=b$ and $\Theta_1[2,3]=bb$,
by Theorem 5.2(1),
$$L(\omega\omega,1)=L(\omega\omega\omega,1)=L(\omega,2)=|\mu_1|+L(E_{m,1},2)=|\mu_1|+2^m+1,$$
where $|\mu_1|+|\mu_2|=|E_{m,1}|-|\omega|=2^{m-1}-1$, so $0\leq|\mu_1|\leq2^{m-1}-1$.
The property holds.
\end{proof}


\vspace{0.5cm}

\noindent\textbf{\large{Acknowledgments}}

\vspace{0.4cm}

The research is supported by the Grant NSFC No.11271223, No.11371210 and No.11326074.

\end{CJK*}
\end{document}